\documentclass[10pt,draft]{article}

\usepackage{amsmath,amsthm,amsfonts,amssymb, color}
\usepackage[left=46mm,top=47mm,right=40mm,bottom=60mm]{geometry} 
\usepackage[T1]{fontenc}
\usepackage[utf8]{inputenc}
\usepackage{authblk}

\makeatletter
\renewcommand\@biblabel[1]{}
\makeatother

\newtheorem{thm}{Theorem}[section]
\newtheorem{cor}[thm]{Corollary}
\newtheorem{lem}[thm]{Lemma}
\newtheorem{rem}[thm]{Remark}

\title{\textbf{\large Existence, Uniqueness and Convergence of Simultaneous Distributed-Boundary Optimal Control Problems}}
\author[1]{\textbf{\normalsize Claudia M. Gariboldi}}
\author[2]{\textbf{\normalsize Domingo A. Tarzia}}
\affil[1]{\normalsize Departamento de Matem\'atica, FCEFQyN, Univ.
Nac. de R\'io Cuarto, Ruta 36 Km 601, 5800 R\'io Cuarto, Argentina.
E-mail: cgariboldi@exa.unrc.edu.ar} \affil[2]{\normalsize
Departamento de Matem\'atica-CONICET, FCE, Univ. Austral, Paraguay
1950, S2000FZF Rosario, Argentina. E-mail: DTarzia@austral.edu.ar}

\begin{document}
\date{}
\maketitle

\begin{abstract}
We consider a steady-state heat conduction problem $P$ for the
Poisson equation with mixed boundary conditions in a bounded
multidimensional domain $\Omega$. We also consider a family of
problems $P_{\alpha}$ for the same Poisson equation with mixed
boundary conditions being $\alpha>0$ the heat transfer coefficient
defined on a portion $\Gamma_{1}$ of the boundary. We formulate
simultaneous \emph{distributed and Neumann boundary} optimal control
problems on the internal energy $g$ within $\Omega$ and the heat
flux $q$, defined on the complementary portion $\Gamma_{2}$ of the
boundary of  $\Omega$ for quadratic cost functional. Here the
control variable is the vector $(g,q)$. We prove existence and
uniqueness of the optimal control
$(\overline{\overline{g}},\overline{\overline{q}})$ for the system
state of $P$, and
$(\overline{\overline{g}}_{\alpha},\overline{\overline{q}}_{\alpha})$
for the system state of $P_{\alpha}$, for each $\alpha>0$, and we
give the corresponding optimality conditions. We prove strong
convergence, in suitable Sobolev spaces, of the vectorial optimal
controls, system and adjoint states governed by the problems
$P_{\alpha}$ to the corresponding vectorial optimal control, system
and adjoint states governed by the problem $P$, when the parameter
$\alpha$ goes to infinity. We also obtain estimations between the
solutions of these vectorial optimal control problems and the
solution of two scalar optimal control problems characterized by
fixed $g$ (with boundary optimal control $\overline{q}$) and fixed
$q$ (with distributed optimal control $\overline{g}$), respectively,
for both cases $\alpha>0$ and $\alpha=\infty$.

\end{abstract}

\textbf{keywords:} Simultaneous optimal control problems, mixed
elliptic problems, optimality condition, elliptic variational
equalities, vectorial optimal control problems.

{\thispagestyle{empty}} 

\section{\large Introduction}
We consider a bounded domain $\Omega $ in ${\mathbb R}^{n}$ whose
regular boundary $\Gamma$ consist of the union of the two disjoint
portions $\Gamma _{1}$ and $\Gamma _{2}$ with $med(\Gamma_{i})>0$
for $i=1,2$. We consider the following steady-state
heat conduction problems $P$ and $P_{\alpha }$ (for each parameter value
$\alpha >0)$ respectively, with mixed boundary conditions:
\begin{equation}
-\Delta u=g\,\ \text{in }\Omega \ \ \,\,\,\,\,\,u\big|_{\Gamma
_{1}}=b\,\,\,\,\,\,\,\,-\frac{\partial u}{\partial n}\big|_{\Gamma
_{2}}=q \label{P}
\end{equation}
\begin{equation}
-\Delta u=g\,\ \text{in }\Omega \ \ \,\,\,\,\,\,-\frac{\partial u}{\partial n%
}\big|_{\Gamma _{1}}=\alpha (u-b)\,\,\,\,\,\,\,\,-\frac{\partial
u}{\partial n}\big|_{\Gamma _{2}}=q  \label{Palfa}
\end{equation}
where $g$ is the internal energy in $\Omega $, $b>0$ is the
temperature on $ \Gamma_{1}$ for (\ref{P}) and the temperature of
the external neighborhood of $\Gamma_{1}$ for (\ref{Palfa}), $q$ is
the heat flux on $\Gamma_{2}$ and $\alpha >0$ is the heat transfer
coefficient on $\Gamma_{1}$ (Newton law or Robin condition on
$\Gamma _{1}$). The following hypothesis: $g\in
L^{2}(\Omega )$, $q\in L^{2}(\Gamma_{2})$ and $b\in
H^{\frac{1}{2}}(\Gamma _{1})$ is assumed to hold. Problems (\ref{P}) and (\ref{Palfa})
can be considered as the steady-state Stefan problem for suitable
data $q$, $g$ and $b$, see Tarzia (1979) or Tabacman and Tarzia
(1989).

We denote by $u_{(g,q)}$ and $u_{(\alpha,g,q)}$ the unique solutions
of the elliptic problems (\ref{P}) and (\ref{Palfa}), respectively,
whose variational formulations are given, as in
 Kinderlehrer and Stampacchia (1980), by:
\begin{equation}
a(u_{(g,q)},v)=L_{(g,q)}(v),\text{ }\forall v\in V_{0},\text{
}u_{(g,q)}\in K \label{Pvariacional}
\end{equation}
\begin{equation}
a_{\alpha }(u_{(\alpha,g,q)},v)=L_{(\alpha,g,q)}(v),\text{ }\forall v\in V,\text{ }%
u_{(\alpha,g,q)}\in V  \label{Palfavariacional}
\end{equation}
where
\[
V=H^{1}(\Omega ),\quad V_{0}=\{v\in V/\,v\big|_{\Gamma
_{1}}=0\},\quad K=v_{0}+V_{0},
\]
\[
R=L^{2}(\Gamma_{1}),\quad H=L^{2}(\Omega ),\quad Q=L^{2}(\Gamma_{2})
\]
for $v_{0}\in V$ given, with $v_{0}\big|_{\Gamma _{1}}=b$ and
\begin{equation}
(g,h)_{H}=\int_{\Omega }gh\,dx;\quad (q,\eta )_{Q}=\int_{\Gamma
_{2}}q\eta \,d\gamma , \quad (b,v)_{R}=\int_{\Gamma _{1}}bv\,d\gamma
\nonumber
\end{equation}
\[
a(u,v)=\int_{\Omega }\nabla u.\nabla vdx;\,\,\,\ \ a_{\alpha
}(u,v)=a(u,v)+\alpha \,\,(u,v)_{R}\,\,
\]
\[
L_{(g,q)}(v)=(g,v)_{H}-(q,v)_{Q};\quad
L_{(\alpha,g,q)}(v)=\,L_{(g,q)}(v)+\alpha \,\,(b,v)_{R}.
\]

The bilinear form $a$ is coercive on $V_{0}$, with coerciveness
constant $\lambda >0$ and the bilinear form $a_{\alpha}$ is coercive
on $V$ with coerciveness constant
$\lambda_{\alpha}=\lambda_{1}min(1,\alpha)$, where $\lambda_{1}>0$
is the coerciveness constant for the bilinear form $a_{1}$, see
Kinderlehrer and Stampacchia (1980) or Tabacman and Tarzia (1989).

We formulate the following simultaneous \emph{distributed and
Neumann boundary} optimal control problems, see Lions (1968) or
Tr\"{o}ltzsch (2010):
\begin{equation}
\text{Find }\quad
(\overline{\overline{g}},\overline{\overline{q}})\in H\times
U_{ad}\quad\text{ such that }\quad
J(\overline{\overline{g}},\overline{\overline{q}})=\min\limits_{g\in
H,q\in U_{ad}} \text{ }J(g,q)  \label{PControl}
\end{equation}
\begin{equation}
\text{Find}\quad(\overline{\overline{g}}_{\alpha},\overline{\overline{q}}_{\alpha})\in
H\times U_{ad}\quad\text{such that}\quad
J_{\alpha}(\overline{\overline{g}}_{\alpha},\overline{\overline{q}}_{\alpha})=\min\limits_{g\in
H,q\in U_{ad}}J_{\alpha}(g,q) \label{PControlalfa}
\end{equation}
with $U_{ad}=\left\{ q\in Q:q\geq 0\text{ on }\Gamma _{2}\right\}$
and the cost functionals $J:H\times Q{\rightarrow}{\mathbb
R}_{0}^{+}$ and $J_{\alpha }:H\times Q{\rightarrow}{\mathbb
R}_{0}^{+}$ being given by:
\begin{equation}
J(g,q)=\frac{1}{2}\left\| u_{(g,q)}-z_{d}\right\|
_{H}^{2}+\frac{M_{1}}{2}\left\|
g\right\|_{H}^{2}+\frac{M_{2}}{2}\left\| q\right\|_{Q}^{2}
\label{J}
\end{equation}
\begin{equation}
J_{\alpha}(g,q)=\frac{1}{2}\left\| u_{(\alpha,g,q)}-z_{d}\right\|
_{H}^{2}+\frac{M_{1}}{2}\left\|
g\right\|_{H}^{2}+\frac{M_{2}}{2}\left\| q\right\| _{Q}^{2}
\label{Jalfa}
\end{equation}
where $z_{d}\in H$, $u_{(g,q)}$ and $u_{(\alpha,g,q)}$ are the
unique solutions of the elliptic variational equalities
(\ref{Pvariacional}) and (\ref{Palfavariacional}) respectively, and
the positive constants $M_{1}$ and $M_{2}$ are given. We remark that we denote the control variables by $g$ and $q$, 
 these two variables corresponding usually the internal energy and the heat flux respectively, in heat transfer problems.

The use of the variational equality theory in connection with
optimization and optimal control problems was done in Belgacem, El
Fekih and Metoui (2003), Bensoussan (1974), Casas (1986), Casas and
Raymond (2006), Kirchner, Meidner and Vexler (2011), Mignot and Puel
(1984).

In Section 2, we obtain the existence and uniqueness of the
vectorial optimal control
$(\overline{\overline{g}},\overline{\overline{q}})$ of the problem
(\ref{PControl}) and of the vectorial optimal control
$(\overline{\overline{g}}_{\alpha},\overline{\overline{q}}_{\alpha})$
of the problem (\ref{PControlalfa}), for each $\alpha>0$. We also
give the optimality conditions in relation to the adjoint state
$p_{(\overline{\overline{g}},\overline{\overline{q}})}$ for
(\ref{PControl}) and
$p_{(\alpha,\overline{\overline{g}}_{\alpha},\overline{\overline{q}}_{\alpha})}$
for (\ref{PControlalfa}).

In Section 3, we obtain estimations between the first component of
the simultaneous optimal control $\overline{\overline{g}}$ and the
scalar optimal control $\overline{g}$ studied in Gariboldi and
Tarzia (2003) (see optimization problem (\ref{PControlJ1})), and the
second component of the simultaneous optimal control
$\overline{\overline{q}}$ and the scalar optimal control
$\overline{q}$ studied in Gariboldi and Tarzia (2008) (see
optimization problem (\ref{PControlJ2})). In the optimal control
problems (\ref{PControl}) and (\ref{PControlalfa}) we have
considered two control variables simultaneously, that is the
solution is a vectorial optimal control, while that in the optimal
control problems, given in Gariboldi and Tarzia (2003) and Gariboldi
and Tarzia (2008) respectively, we have considered only one control
variable, namely the solutions are scalar optimal controls.
Moreover, we characterize the optimal control
$(\overline{\overline{g}},\overline{\overline{q}})$ as a fixed point
on $H\times Q$ for a suitable operator $W$. In similar way, we
obtain estimations for the optimal controls of the problems
$P_{\alpha}$, for each $\alpha>0$, and we characterize the optimal
control
$(\overline{\overline{g}}_{\alpha},\overline{\overline{q}}_{\alpha})$
as a fixed point on $H\times Q$ for a suitable operator
$W_{\alpha}$.

In Section 4, we prove the strongly convergence, in suitable Sobolev
spaces, of the optimal controls
$(\overline{\overline{g}}_{\alpha},\overline{\overline{q}}_{\alpha})$
of the problems (\ref{PControlalfa}) to the optimal control
$(\overline{\overline{g}},\overline{\overline{q}})$ of the problem
(\ref{PControl}), of the system states
$u_{(\alpha,\overline{\overline{g}}_{\alpha},\overline{\overline{q}}_{\alpha})}$
to the system state
$u_{(\overline{\overline{g}},\overline{\overline{q}})}$ and of the
adjoint sates
$p_{(\alpha,\overline{\overline{g}}_{\alpha},\overline{\overline{q}}_{\alpha})}$
to the adjoint state
$p_{(\overline{\overline{g}},\overline{\overline{q}})}$, when the
parameter $\alpha$ goes to infinity. We also prove the convergence
of the corresponding cost functional when $\alpha$ goes to infinity.

This asymptotic behavior can be considered very important in the
optimal control of heat transfer problems because the Dirichlet
boundary condition, given in (\ref{P}), can be approximated by the
relevant physical condition given by the Newton law or the Robin
boundary condition given in (\ref{Palfa}), see Carslaw and Jaeger
(1959). Therefore, the goal of this paper is to approximate a
Dirichlet boundary condition in a vectorial optimal control problem,
governed by an elliptic variational equality, by a Robin boundary
condition in a family of vectorial optimal control problems,
governed also by elliptic variational equalities, for a large
positive coefficient $\alpha$. Particular cases of our results can
be considered the ones given in Gariboldi and Tarzia (2003) when the
scalar control variable is the internal energy $g$ for both state
systems (\ref{P}) and (\ref{Palfa}), and in Gariboldi and Tarzia
(2008) when the scalar control variable is the heat flux $q$ on the
boundary $\Gamma_{2}$ for both state systems (\ref{P}) and
(\ref{Palfa}). In Belgacem, El Fekih and Metoui (2003) the control
variable is the temperature $b$ on the boundary $\Gamma_{1}$ for the
state system (\ref{P}), and the temperature of the external
neighborhood $b$ on the boundary $\Gamma_{1}$ for the state systems
(\ref{Palfa}), this being essentially different with respect to the present vectorial
optimal control problems.

\section{\large Existence and Uniqueness of Optimal Controls}

\subsection{Problem $P$ and its Optimal Control Problem}

Let $C:H\times Q\,{\rightarrow }\,V_{0}$ be the application defined
by $C(g,q)=u_{(g,q)}-u_{(0,0)}$ where $u_{(0,0)}$ is the solution of
the problem (\ref{P}) for $g=0$ and $q=0$. We define, in the way similar 
to Gariboldi and Tarzia (2003), Gariboldi and Tarzia (2008) and
Lions (1968), the applications $\Pi:(H\times Q)\times (H\times
Q)\,{\rightarrow}\,{\mathbb R}$, and $L:H\times Q{\rightarrow
}{\mathbb R}$ as follows:
\[
\Pi((g,q),(h,\eta))=(C(g,q),C(h,\eta))_{H}+M_{1}(g,h)_{H}+M_{2}(q,\eta
)_{Q}
\]
\[
L((g,q))=(C(g,q),z_{d}-u_{(0,0)})_{H},\quad \forall \,(g,q),(h,\eta)
\in H\times Q.
\]

For each $(g,q)\in H\times Q$, we define the adjoint state
$p_{(g,q)}$ corresponding to the problem (\ref{P}), as the unique
solution of the mixed elliptic problem
\begin{equation}
-\Delta p=u-z_{d}\,\ \text{in }\Omega, \ \
\,\,\,\,\,\,p\big|_{\Gamma _{1}}=0,\,\,\,\,\,\,\,\,\frac{\partial
p}{\partial n}\big|_{\Gamma _{2}}=0 \label{Padjunto}
\end{equation}
whose variational formulation is given by
\begin{equation}
a(p_{(g,q)},v)=(u_{(g,q)}-z_{d},v)_{H},\text{ }\forall v\in
V_{0},\text{ }p_{(g,q)}\in V_{0}. \label{Padjuntovariac}
\end{equation}
and we have the following properties.

\begin{thm}
a) $J$ is a coercive and strictly convex functional on $H\times Q$.

b) The adjoint state $p_{(g,q)}$ satisfy, $\forall (h,\eta)\in H\times Q$:
\[
a(p_{(g,q)},C(h,\eta))=(C(h,\eta),u_{(g,q)}-z_{d})_{H}=(h,p_{(g,q)})_{H}-(\eta,p_{(g,q)})_{Q}
\]

c) $J$ is G\^{a}teaux differentiable and $J'$ is given by, $\forall
(h,\eta)\in H\times Q$:
\begin{equation}
J'(g,q)(h-g,\eta-q)=\Pi((g,q),(h-g,\eta-q))-L(h-g,\eta-q)\label{jotaprima}
\end{equation}

d) There exists a unique solution
$(\overline{\overline{g}},\overline{\overline{q}})\in H\times
U_{ad}$ of the vectorial optimal control problem (\ref{PControl})
and its optimality condition is given by, $\forall (h,\eta)\in
H\times U_{ad}$:
\[
(h-\overline{\overline{g}},p_{(\overline{\overline{g}},\overline{\overline{q}})}+M_{1}\overline{\overline{g}})_{H}+(\eta-\overline{\overline{q}},M_{2}\overline{\overline{q}}-p_{(\overline{\overline{g}},\overline{\overline{q}})})_{Q}\geq
0.
\]
\end{thm}

\begin{proof}
(a) It is sufficient to prove that, $\forall
(g_{2},q_{2}),(g_{1},q_{1})\in H\times Q$ and $\forall t\in[0,1]$,
we have, see Lions (1968), Boukrouche and Tarzia (2007) or
Tr\"{o}ltzsch (2010):
\begin{equation}
u_{((1-t)g_{2}+tg_{1},(1-t)q_{2}+tq_{1})}=(1-t)u_{(g_{2},q_{2})}+tu_{(g_{1},q_{1})}\label{convex}
\end{equation}
and
\[
(1-t)J(g_{2},q_{2})+tJ(g_{1},q_{1})-J((1-t)(g_{2},q_{2})+t(g_{1},q_{1}))=
\]
\begin{equation}
=\frac{t(1-t)}{2}\left[\|u_{(g_{2},q_{2})}-u_{(g_{1},q_{1})}\|^{2}_{H}+
M_{1}\|g_{2}-g_{1}\|^{2}_{H}+M_{2}\|q_{2}-q_{1}\|^{2}_{Q}\right]\geq\nonumber
\end{equation}
\begin{equation}
\geq\frac{Mt(1-t)}{2}\|(g_{2}-g_{1},q_{2}-q_{1})\|^{2}_{H\times
Q}. \label{EstimJ}
\end{equation}
and
\[
a(p_{(g,q)},C(h,\eta))=(-\Delta
p_{(g,q)},u_{(g,q)}-u_{(0,0)})_{H}=
\]
\begin{equation}
=(h,p_{(g,q)})_{H}-(\eta,p_{(g,q)})_{Q}\label{a}
\end{equation}
where
\[
\|(g,q)\|^{2}_{H\times Q}=\|g\|^{2}_{H}+\|q\|^{2}_{Q},\quad M=Min(M_{1},M_{2})>0.
\]
\end{proof}

\subsection{Problem $P_{\alpha}$ and its Optimal Control Problem}

Let $C_{\alpha}:H\times Q{\rightarrow}V$ be the application defined
by $C_{\alpha}(g,q)=u_{(\alpha,g,q)}-u_{(\alpha,0,0)}$ where
$u_{(\alpha,0,0)}$ is the solution of the problem (\ref{Palfa}) for
$g=0$ and $q=0$. We define the applications $\Pi_{\alpha}:(H\times
Q)\times (H\times Q){\rightarrow}{\mathbb R}$ and
$L_{\alpha}:H\times Q{\rightarrow}{\mathbb R}$ by the following
expressions, $\forall \,\,(g,q),(h,\eta) \in H\times Q$:
\[
\Pi_{\alpha}((g,q),(h,\eta))=(C_{\alpha}(g,q),C_{\alpha}(h,\eta))_{H}+M_{1}(g,h)_{H}+M_{2}(q,\eta
)_{Q},
\]
\[
L_{\alpha}(q)=(C_{\alpha}(g,q),z_{d}-u_{(\alpha,0,0)})_{H}.
\]

For each $(g,q)\in H\times Q$ and $\alpha >0$, we define the adjoint
state $p_{(\alpha,g,q)}$ corresponding to the problem (\ref{Palfa}),
as the unique solution of the mixed elliptic problem
\begin{equation}
-\Delta p=u-z_{d}\,\ \text{in }\Omega, \ \
\,\,\,\,\,\,-\frac{\partial p}{\partial n}\big|_{\Gamma _{1}}=\alpha
p,\,\,\,\,\,\,\,\,\frac{\partial p}{\partial n}\big|_{\Gamma _{2}}=0
\label{Padjuntoalfa}
\end{equation}
whose variational formulation is given by
\begin{equation}
a_{\alpha}(p_{(\alpha,g,q)},v)=(u_{(\alpha,g,q)}-z_{d},v)_{H},\text{
}\forall v\in V,\text{ }p_{(\alpha,g,q)}\in V.
\label{Padjuntovariacalfa}
\end{equation}

We can obtain similar properties to Theorem 2.1, following
Boukrouche and Tarzia \linebreak (2007), Kinderlehrer and
Stampacchia (1980), Lions (1968) or Tr\"{o}ltzsch (2010).

\begin{thm} We have, for each $\alpha >0$, the following
properties:

a) $J_{\alpha}$ is a coercive and strictly convex functional on $H\times Q$.

b) The adjoint state $p_{(\alpha,g,q)}$ satisfy, $\forall (h,\eta)\in H\times Q$:
\[
a_{\alpha}(p_{(\alpha,g,q)},C_{\alpha}(h,\eta))=(C_{\alpha}(h,\eta),u_{(\alpha,g,q)}-z_{d})_{H}=(h,p_{(\alpha,g,q)})_{H}-(\eta,p_{(\alpha,g,q)})_{Q}.
\]

c) $J_{\alpha}$ is G\^{a}teaux differentiable and $J'_{\alpha}$ is
given by, $\forall (h,\eta)\in H\times Q$:
\begin{equation}
J'_{\alpha}(g,q)(h-g,\eta-q)=\Pi_{\alpha}((g,q),(h-g,\eta-q))-L_{\alpha}(h-g,\eta-q)
\end{equation}

d) There exists a unique solution
$(\overline{\overline{g}}_{\alpha},\overline{\overline{q}}_{\alpha})\in
H\times U_{ad}$ of the vectorial optimal control pro\-blem
(\ref{PControlalfa}) and its optimality condition is given by,
$\forall (h,\eta)\in H\times U_{ad}$:
\[
(h-\overline{\overline{g}}_{\alpha},p_{(\alpha,\overline{\overline{g}}_{\alpha},\overline{\overline{q}}_{\alpha})}+M_{1}\overline{\overline{g}}_{\alpha})_{H}+(\eta-\overline{\overline{q}}_{\alpha},M_{2}\overline{\overline{q}}_{\alpha}-p_{(\alpha,\overline{\overline{g}}_{\alpha},
\overline{\overline{q}}_{\alpha})})_{Q}\geq 0.
\]
\end{thm}

\section{\large Estimations}

\subsection{Estimations with respect to the problem $P$}

We consider the scalar distributed optimal control problem:
\begin{equation}
\text{Find }\quad\overline{g}\in H\quad\text{ such that }\quad
J_{1}(\overline{g})=\min\limits_{g\in H} \text{ }J_{1}(g),
\,\,\,\text{ for fixed } q\in Q,\label{PControlJ1}
\end{equation}
and the scalar Neumann boundary optimal control problem:
\begin{equation}
\text{Find }\quad \overline{q}\in U_{ad}\quad\text{ such that }\quad
J_{2}(\overline{q})=\min\limits_{q\in U_{ad}}\,J_{2}(q),
\,\,\,\text{ for fixed } g\in H, \label{PControlJ2}
\end{equation}
where $J_{1}$ is the cost functional given in Gariboldi and Tarzia (2003) plus
the constant $\frac{M_{2}}{2}\left\| q\right\|_{Q}^{2}$ and
$J_{2}$ is the functional given in Gariboldi and Tarzia (2008) plus the constant
$\frac{M_{1}}{2}\left\| g\right\|_{H}^{2}$, that is,
$J_{1}:H{\rightarrow}{\mathbb R}_{0}^{+}$ and
$J_{2}:Q{\rightarrow}{\mathbb R}_{0}^{+}$, are given by:
\begin{equation}
J_{1}(g)=\frac{1}{2}\left\|u_{g}-z_{d}\right\|
_{H}^{2}+\frac{M_{1}}{2}\left\|
g\right\|_{H}^{2}+\frac{M_{2}}{2}\left\| q\right\|_{Q}^{2},\quad
(\text{ fixed }q\in Q)\label{J1}
\end{equation}
\begin{equation}
J_{2}(q)=\frac{1}{2}\left\| u_{q}-z_{d}\right\|
_{H}^{2}+\frac{M_{2}}{2}\left\|
q\right\|_{Q}^{2}+\frac{M_{1}}{2}\left\| g\right\|_{H}^{2},\quad
(\text{ fixed } g\in H)\label{J2}
\end{equation}
where $u_{g}$ and $u_{q}$ are the unique solutions of the problem
(\ref{P}) for fixed $q$ and $g$ data, respectively.

\begin{rem} The functionals $J$, $J_{1}$ and $J_{2}$ satisfy the elemental estimations
\begin{equation}
J(\overline{\overline{g}},\overline{\overline{q}})\leq
J_{1}(\overline{g}),\,\, \forall q\in Q \quad \text{and}\quad
J(\overline{\overline{g}},\overline{\overline{q}})\leq
J_{2}(\overline{q}),\,\, \forall g\in H.\nonumber
\end{equation}
\end{rem}

In the next theorem we will obtain estimations between the solution
of the scalar distributed optimal control problem (\ref{PControlJ1})
with the first component of the solution of the vectorial
distributed and Neumann boundary optimal control problem
(\ref{PControl}), and between the solution of the scalar Neumann
boundary optimal control problem (\ref{PControlJ2}) with the second
component of the solution of the vectorial distributed and Neumann
boundary optimal control problem (\ref{PControl}).

\begin{thm} If $(\overline{\overline{g}},\overline{\overline{q}})\in H\times U_{ad}$ is
the unique solution of the vectorial optimal control problem
(\ref{PControl}), and $\overline{g}$  and $\overline{q}$ are the
unique solutions of the scalar optimal control problems
(\ref{PControlJ1}) and (\ref{PControlJ2}) respectively, then:
\begin{equation}
\|\overline{q}-\overline{\overline{q}}\|_{Q}\leq
\frac{\|\gamma_{0}\|}{\lambda
M_{2}}\|u_{(\overline{\overline{g}},\overline{\overline{q}})}-u_{(\overline{\overline{g}},\overline{q})}\|_{H}\label{estim-q}
\end{equation}
\begin{equation}
\|\overline{g}-\overline{\overline{g}}\|_{H}\leq \frac{1}{\lambda
M_{1}}\|u_{(\overline{\overline{g}},\overline{\overline{q}})}-u_{(\overline{g},\overline{\overline{q}})}\|_{H},\label{estim-g}
\end{equation}
where $\gamma_{0}$ is the trace operator.
\end{thm}
\begin{proof}
For $g=\overline{\overline{g}}$, from the optimality condition for
$\overline{q}$ , see Gariboldi and Tarzia (2008), we have
\begin{equation}
(M_{2}\overline{q}-p_{(\overline{\overline{g}},\overline{q})},\eta
-\overline{q})_{Q}\geq 0, \quad \forall \eta \in U_{ad}.\label{a}
\end{equation}
If we take $h=\overline{\overline{g}}\in H$ in the optimality
condition for $(\overline{\overline{g}},\overline{\overline{q}})$,
we obtain
\begin{equation}
(M_{2}\overline{\overline{q}}-p_{(\overline{\overline{g}},\overline{\overline{q}})},\eta
-\overline{\overline{q}})_{Q}\geq 0, \quad \forall \eta \in U_{ad}.
\label{b}
\end{equation}
Now, taking $\eta =\overline{\overline{q}}\in U_{ad}$ in (\ref{a})
and $\eta =\overline{q}\in U_{ad}$ in (\ref{b}), we obtain
\begin{equation}
(M_{2}(\overline{q}
-\overline{\overline{q}})+(p_{(\overline{\overline{g}},\overline{\overline{q}})}-p_{(\overline{\overline{g}},\overline{q})}),
\overline{\overline{q}}-\overline{q})_{Q}\geq 0,\nonumber
\end{equation}
and by using
$\|p_{(\overline{\overline{g}},\overline{\overline{q}})}-p_{(\overline{\overline{g}},\overline{q})}\|_{V}
\leq \frac{1}{\lambda}
\|u_{(\overline{\overline{g}},\overline{\overline{q}})}-u_{(\overline{\overline{g}},\overline{q})}\|_{H}$
we deduce
\[
\|\overline{q}-\overline{\overline{q}}\|_{Q}\leq
\frac{\|\gamma_{0}\|}{M_{2}}\|p_{(\overline{\overline{g}},\overline{\overline{q}})}-p_{(\overline{\overline{g}},\overline{q})}\|_{V}
\leq \frac{\|\gamma_{0}\|}{\lambda
M_{2}}\|u_{(\overline{\overline{g}},\overline{\overline{q}})}-u_{(\overline{\overline{g}},\overline{q})}\|_{H}
\]
therefore the estimation (\ref{estim-q}) holds. Similarly, the estimation (\ref{estim-g}) holds.
\end{proof}

When we consider the vectorial distributed and Neumann boundary
optimal control problem (\ref{PControl}) without restrictions, i.e.
$U_{ad}=Q$, then we can characterize the solution of
(\ref{PControl}) by using the fixed point theory.

Let $W:H\times Q\rightarrow H\times Q$ be the operator defined by
\begin{equation}
W(g,q)=(-\frac{1}{M_{1}}p_{(g,q)},\frac{1}{M_{2}}p_{(g,q)}).\label{Woperador}
\end{equation}

\begin{thm}
There exists a positive constant
$C_{0}=C_{0}(\lambda,\gamma_{0},M_{1},M_{2})$ such that, \linebreak
$\forall (g_{1},q_{1}), (g_{2},q_{2})\in H\times Q$:
\begin{equation}
\|W(g_{2},q_{2})-W(g_{1},q_{1})\|_{H\times Q}\leq
C_{0}\|(g_{2},q_{2})-(g_{1},q_{1})\|_{H\times Q}\label{contraction}
\end{equation}
and $W$ is a contraction operator if and only if data satisfy the
following condition:
\begin{equation}
C_{0}=\frac{\sqrt{2}}{\lambda^{2}}\sqrt{\frac{1}{M_{1}^{2}}+\frac{\|\gamma_{0}\|^{2}}
{M_{2}^{2}}}(1+\|\gamma_{0}\|)<1.\label{Contrac}
\end{equation}
\end{thm}

\begin{proof}
By using the estimations, $\forall (g_{1},q_{1}),(g_{2},q_{2}) \in H\times Q$:
\begin{equation}
\|u_{(g_{1},q_{1})}-u_{(g_{2},q_{2})}\|_{V}\leq \frac{1}{\lambda}
(\|g_{2}-g_{1}\|_{H}+\|\gamma_{0}\|\|q_{2}-q_{1}\|_{Q}),
\label{estim-a}
\end{equation}
\begin{equation}
\|p_{(g_{1},q_{1})}-p_{(g_{2},q_{2})}\|_{V}\leq \frac{1}{\lambda}
\|u_{(g_{1},q_{1})}-u_{(g_{2},q_{2})}\|_{H}\label{estim-d}
\end{equation}
we obtain
\begin{equation}
\|W(g_{2},q_{2})-W(g_{1},q_{1})\|^{2}_{H\times Q}\leq
(\frac{1}{M_{1}^{2}}+\frac{\|\gamma_{0}\|^{2}}{M_{2}^{2}})\frac{1}{\lambda^{4}}[\|g_{2}-g_{1}\|_{H}+\|\gamma_{0}\|\|q_{2}-q_{1}\|_{Q}]^{2}.\nonumber
\end{equation}
Then (\ref{contraction}) holds and the operator $W$ is a contraction
if and only if data satisfy inequality (\ref{Contrac}).
\end{proof}

\begin{cor}
If data satisfy inequality (\ref{Contrac}) then the unique solution
$(\overline{\overline{g}},\overline{\overline{q}})\in H\times Q$ of
the vectorial optimal control problem (\ref{PControl}) can be
obtained as the unique fixed point of the operator $W$, that is:
\[
W(\overline{\overline{g}},\overline{\overline{q}})=(-\frac{1}{M_{1}}p_{(\overline{\overline{g}},\overline{\overline{q}})},\frac{1}{M_{2}}p_{(\overline{\overline{g}},\overline{\overline{q}})})=(\overline{\overline{g}},\overline{\overline{q}}).
\]
\end{cor}

\subsection{Estimations with respect to the problem $P_{\alpha}$}

For each $\alpha >0$, we consider the scalar optimal control
problems:
\begin{equation}
\text{Find }\quad \overline{g}_{\alpha}\in H\quad\text{ such that
}\quad J_{1\alpha}(\overline{g}_{\alpha})=\min\limits_{g\in H}
\text{ }J_{1\alpha}(g),  \label{PControlalfaJ1}
\end{equation}
\begin{equation}
\text{Find }\quad \overline{q}_{\alpha}\in U_{ad}\quad\text{ such
that }\quad J_{2\alpha}(\overline{q}_{\alpha})=\min\limits_{q\in
U_{ad}}\,J_{2\alpha}(q),  \label{PControlalfaJ2}
\end{equation}
where $J_{1\alpha}:H{\rightarrow}{\mathbb R}_{0}^{+}$ and $J_{2\alpha
}:Q{\rightarrow}{\mathbb R}_{0}^{+}$, are given by:
\begin{equation}
J_{1\alpha }(g)=\frac{1}{2}\left\|u_{\alpha g}-z_{d}\right\|
_{H}^{2}+\frac{M_{1}}{2}\left\|
g\right\|_{H}^{2}+\frac{M_{2}}{2}\left\| q\right\|_{Q}^{2},\quad
(\text{ fixed }q\in Q)\label{J1}
\end{equation}
\begin{equation}
J_{2\alpha }(q)=\frac{1}{2}\left\| u_{\alpha q}-z_{d}\right\|
_{H}^{2}+\frac{M_{2}}{2}\left\|
q\right\|_{Q}^{2}+\frac{M_{1}}{2}\left\| g\right\|_{H}^{2},\quad
(\text{ fixed }g\in H)\label{J2}
\end{equation}
where $J_{1\alpha}$ is the functional studied in Gariboldi and Tarzia (2003) plus
the constant $\frac{M_{2}}{2}\left\| q\right\|_{Q}^{2}$,
$J_{2\alpha}$ is the functional studied in Gariboldi and Tarzia (2008) plus the
constant $\frac{M_{1}}{2}\left\| g\right\|_{H}^{2}$, and the
system states $u_{\alpha g}$ and $u_{\alpha q}$ are the unique
solutions of the problem (\ref{Palfa}) for fixed data $q$ and $g$,
respectively.

\begin{rem} The functionals $J_{\alpha}$, $J_{1\alpha
}$ and $J_{2\alpha }$ satisfy the estimations
\[
J_{\alpha}(\overline{\overline{g}}_{\alpha},\overline{\overline{q}}_{\alpha})\leq
J_{1\alpha }(\overline{g}_{\alpha}),\, \forall q\in Q\,\,
\text{and}\,\,
J_{\alpha}(\overline{\overline{g}}_{\alpha},\overline{\overline{q}}_{\alpha})\leq
J_{2\alpha }(\overline{q}_{\alpha}),\, \forall g\in H.
\]
\end{rem}

Estimations between the solution of the scalar distributed optimal
control problem (\ref{PControlalfaJ1}) with respect to the first
component of the solution of the vectorial distributed and Neumann
boundary optimal control problem (\ref{PControlalfa}), and
estimations between the solution of the scalar Neumann boundary
optimal control problem (\ref{PControlalfaJ2}) with respect to the
second component of the solution of the vectorial distributed and
Neumann boundary optimal control problem (\ref{PControlalfa}) are
given in the next theorem whose proof is omitted.

\begin{thm} If $(\overline{\overline{g}}_{\alpha},\overline{\overline{q}}_{\alpha})\in H\times U_{ad}$ is the
unique solution of the vectorial optimal control problem
(\ref{PControlalfa}), and $\overline{g}_{\alpha}$ and
$\overline{q}_{\alpha}$ are the unique solutions of the scalar
optimal control problems (\ref{PControlalfaJ1}) and
(\ref{PControlalfaJ2}) respectively, then we have the following
estimations
\[
\|\overline{q}_{\alpha}-\overline{\overline{q}}_{\alpha}\|_{Q}\leq
\frac{\|\gamma_{0}\|}{\lambda
M_{2}}\|u_{(\alpha,\overline{\overline{g}}_{\alpha},\overline{\overline{q}}_{\alpha})}-u_{(\alpha,\overline{\overline{g}}_{\alpha},
\overline{q}_{\alpha})}\|_{H}
\]
\[
\|\overline{g}_{\alpha}-\overline{\overline{g}}_{\alpha}\|_{H}\leq
\frac{1}{\lambda
M_{1}}\|u_{(\alpha,\overline{\overline{g}}_{\alpha},\overline{\overline{q}}_{\alpha})}-u_{(\alpha,\overline{g}_{\alpha},\overline{\overline{q}}_{\alpha})}\|_{H}.
\]
\end{thm}

In the way similar to Theorem 3.3, we can now characterize the solution
of the vectorial distributed and Neumann boundary optimal control
problem (\ref{PControlalfa}), without restrictions, proving that a
suitable operator $W_{\alpha}$ is a contraction. It is presented in
the next theorem and the proof is omitted. We define the operator
$W_{\alpha}:H\times Q\rightarrow H\times Q$, for each $\alpha>0$, by
the expression
\begin{equation}
W_{\alpha}(g,q)=(-\frac{1}{M_{1}}p_{(\alpha,g,q)},\frac{1}{M_{2}}p_{(\alpha,g,q)})\label{Walfa}.
\end{equation}

\begin{thm}
$W_{\alpha}$ is a Lipschitz operator over $H\times Q$, that is,
there exists a positive constant
$C_{0\alpha}=C_{0\alpha}(\lambda_{\alpha},\gamma_{0},M_{1},M_{2})$, such
that:
\begin{equation}
\|W_{\alpha}(g_{2},q_{2})-W_{\alpha}(g_{1},q_{1})\|_{H\times
Q}\leq C_{0\alpha}\|(g_{2}-g_{1},q_{2}-q_{1})\|_{H\times Q}
\end{equation}
and $W_{\alpha}$ is a contraction operator if and only if data
satisfy the following inequality:
\begin{equation}
C_{0\alpha}=\frac{\sqrt{2}}{\lambda_{\alpha}^{2}}\sqrt{\frac{1}{M_{1}^{2}}+\frac{\|\gamma_{0}\|^{2}}{M_{2}^{2}}}(1+\|\gamma_{0}\|)<1.\label{contracalfa}
\end{equation}
\end{thm}

\begin{cor}
If data satisfy inequality $C_{0\alpha}<1$, then the unique solution
$(\overline{\overline{g}}_{\alpha},\overline{\overline{q}}_{\alpha})\in
H\times Q$ of the vectorial optimal control problem
(\ref{PControlalfa}) can be obtained as the unique fixed point of
the operator $W_{\alpha}$, that is:
\[
W_{\alpha}(\overline{\overline{g}}_{\alpha},\overline{\overline{q}}_{\alpha})=(-\frac{1}{M_{1}}
p_{(\alpha,\overline{\overline{g}}_{\alpha},\overline{\overline{q}}_{\alpha})},\frac{1}{M_{2}}
p_{(\alpha,\overline{\overline{g}}_{\alpha},\overline{\overline{q}}_{\alpha})})=(\overline{\overline{g}}_{\alpha},\overline{\overline{q}}_{\alpha}).
\]
\end{cor}

\section{\large Convergence when $\alpha\rightarrow +\infty$}

\begin{lem}
For each $\alpha >0$, $(g,q)\in H\times Q$, $b\in
H^{1/2}(\Gamma_{1})$, we have the following limits:
\begin{equation}
i)
\lim\limits_{\alpha\rightarrow\infty}\|u_{(\alpha,g,q)}-u_{(g,q)}\|_{V}=0\qquad
ii)
\lim\limits_{\alpha\rightarrow\infty}\|p_{(\alpha,g,q)}-p_{(g,q)}\|_{V}=0\label{convpfijo}
\end{equation}
\end{lem}
\begin{proof}
We follow in a similar way to the one given in Gariboldi and Tarzia (2003) and
Gariboldi and Tarzia (2008).
\end{proof}

\begin{thm}
i) If $u_{(\overline{\overline{g}},\overline{\overline{q}})}$ and
$u_{(\alpha,\overline{\overline{g}}_{\alpha},\overline{\overline{q}}_{\alpha})}$
are the unique system states corresponding of the vectorial optimal
control problems (\ref{PControl}) and (\ref{PControlalfa})
respectively, then:
\begin{equation}
\lim\limits_{\alpha\rightarrow\infty}\|u_{(\alpha,\overline{\overline{g}}_{\alpha},\overline{\overline{q}}_{\alpha})}-u_{(\overline{\overline{g}},\overline{\overline{q}})}\|_{V}=0.\label{uconvfuerte}
\end{equation}

ii) If $p_{(\overline{\overline{g}},\overline{\overline{q}})}$ and
$p_{(\alpha,\overline{\overline{g}}_{\alpha},\overline{\overline{q}}_{\alpha})}$
are the unique adjoint states corresponding to the vectorial optimal
control problems (\ref{PControl}) and (\ref{PControlalfa})
respectively, then:
\begin{equation}
\lim\limits_{\alpha\rightarrow\infty}\|p_{(\alpha,\overline{\overline{g}}_{\alpha},\overline{\overline{q}}_{\alpha})}-p_{(\overline{\overline{g}},\overline{\overline{q}})}\|_{V}=0.\label{pconvfuerte}
\end{equation}

iii) If $(\overline{\overline{g}},\overline{\overline{q}})$ and
$(\overline{\overline{g}}_{\alpha},\overline{\overline{q}}_{\alpha})$
are the unique solutions of the \emph{simulta\-neous distributed and
Neumann boundary} optimal control problems (\ref{PControl}) and
(\ref{PControlalfa}) respectively, then:
\begin{equation}
\lim\limits_{\alpha\rightarrow\infty}\|(\overline{\overline{g}}_{\alpha},\overline{\overline{q}}_{\alpha})-(\overline{\overline{g}},\overline{\overline{q}})\|_{H\times
Q}=0.\label{controlesconvfuerte}
\end{equation}
\end{thm}

\begin{proof} The proof is given by two step:

\textbf{Step 1.} From the optimal control problem (\ref{PControlalfa}) we deduce that there exist positive constants
$C_{1}$, $C_{2}$ and $C_{3}$, independent of $\alpha$, such that
\begin{equation}
\|u_{(\alpha,\overline{\overline{g}}_{\alpha},\overline{\overline{q}}_{\alpha})}-z_{d}\|_{H}\leq
C_{1},\quad \|\overline{\overline{g}}_{\alpha}\|_{H}\leq C_{2},\quad
\|\overline{\overline{q}}_{\alpha}\|_{Q}\leq C_{3}. \label{cotas}
\end{equation}
Now, if we take
$v=u_{(\alpha,\overline{\overline{g}}_{\alpha},\overline{\overline{q}}_{\alpha})}-u_{(\overline{\overline{g}},\overline{\overline{q}})}\in
V$ in the variational equality (\ref{Palfavariacional}), following
Gariboldi and Tarzia (2003) or Gariboldi and Tarzia (2008), we
obtain, for $\alpha>1$,
$\|u_{(\alpha,\overline{\overline{g}}_{\alpha},\overline{\overline{q}}_{\alpha})}\|_{V}\leq
C_{4}$ where
$C_{4}=C_{4}(C_{2},C_{3},\gamma_{0},u_{(\overline{\overline{g}},\overline{\overline{q}})},\lambda_{1})$
is independent of $\alpha$. Therefore,
\begin{equation}
\exists \mu\in K \,\text{ such that }\,
u_{(\alpha,\overline{\overline{g}}_{\alpha},\overline{\overline{q}}_{\alpha})}\rightharpoonup
\mu\, \text{ weakly in }V,\text{ when }\alpha\rightarrow
+\infty.\label{udebil}
\end{equation}
Taking
$v=p_{(\alpha,\overline{\overline{g}}_{\alpha},\overline{\overline{q}}_{\alpha})}-p_{(\overline{\overline{g}},\overline{\overline{q}})}\in
V$ in the variational equality (\ref{Padjuntovariacalfa}), we obtain
that there exists a positive constant
$C_{5}=C_{5}(C_{1},p_{(\overline{\overline{g}},\overline{\overline{q}})},\lambda_{1})$,
such that
$\|p_{(\alpha,\overline{\overline{g}}_{\alpha},\overline{\overline{q}}_{\alpha})}\|_{V}\leq
C_{5}$ and next
\begin{equation}
\exists \xi\in V_{0} \,\text{ such that }\,
p_{(\alpha,\overline{\overline{g}}_{\alpha},\overline{\overline{q}}_{\alpha})}\rightharpoonup
\xi\, \text{ weakly in }V,\text{ when }\alpha\rightarrow
+\infty.\label{pdebil}
\end{equation}
Moreover, from (\ref{cotas}), we deduce that there exist $f\in Q$
and $h\in H$ such that
\begin{equation}
\overline{\overline{q}}_{\alpha}\rightharpoonup f \text{ weakly in
}Q,\text{ when }\alpha\rightarrow +\infty\label{qdebil}
\end{equation}
\begin{equation}
\overline{\overline{g}}_{\alpha}\rightharpoonup h \text{ weakly in
}H,\text{ when }\alpha\rightarrow +\infty. \label{gdebil}
\end{equation}
For $v\in V_{0}$, taking into account (\ref{udebil}),
(\ref{qdebil}), (\ref{gdebil}) and taking the limit as $\alpha$
goes to infinity, we have that
\begin{equation}
a(\mu,v)=(h,v)_{H}-(f,v)_{Q},\text{ }\forall v\in V_{0},\mu\in K
\label{Pvariacional2}
\end{equation}
and by the uniqueness of the solution of (\ref{Pvariacional}), we get $\mu=u_{hf}$.\\
Now, for $v\in V_{0}$, taking into account (\ref{pdebil}),
with the parameter $\alpha$ going to infinity in the
variational equality (\ref{Padjuntovariacalfa}), we have that
\begin{equation}
a(\xi,v)=(u_{hf}-z_{d},v)_{H},\text{ }\forall v\in V_{0},\xi\in
V_{0} \label{Padjuntovariac2}
\end{equation}
and from the uniqueness of the solution of (\ref{Padjuntovariac}),
we get $\xi=p_{hf}$. Next,
\[
J(h,f)\leq \liminf
\limits_{\alpha\rightarrow\infty}J_{\alpha}(\overline{\overline{g}}_{\alpha},\overline{\overline{q}}_{\alpha})\leq\liminf
\limits_{\alpha\rightarrow\infty}J_{\alpha}(h',f')=
\]
\[
=\lim
\limits_{\alpha\rightarrow\infty}J_{\alpha}(h',f')=J(h',f'),\quad
\forall (h',f')\in H\times Q,
\]
and from the uniqueness of the solution of the problem
(\ref{PControl}), we have that $h=\overline{\overline{g}}$ and
$f=\overline{\overline{q}}$. Therefore, we have proved that
\begin{equation}
u_{(\alpha,\overline{\overline{g}}_{\alpha},\overline{\overline{q}}_{\alpha})}\rightharpoonup
u_{(\overline{\overline{g}},\overline{\overline{q}})} \text{ weakly
in }V,\text{ when }\alpha\rightarrow +\infty\label{udebil2}
\end{equation}
\begin{equation}
p_{(\alpha,\overline{\overline{g}}_{\alpha},\overline{\overline{q}}_{\alpha})}\rightharpoonup
p_{(\overline{\overline{g}},\overline{\overline{q}})} \text{ weakly
in }V,\text{ when }\alpha\rightarrow +\infty.\label{pdebil2}
\end{equation}
\textbf{Step 2.} Taking $h=0$ and $\eta=\overline{\overline{q}}$ in
the optimality condition for the vectorial optimal control problem
(\ref{PControlalfa}), $h=0$ and
$\eta=\overline{\overline{q}}_{\alpha}$ in the optimality condition
for the vectorial optimal control problem (\ref{PControl}), we have
\[
(\overline{\overline{q}}_{\alpha}-\overline{\overline{q}},M_{2}(\overline{\overline{q}}-\overline{\overline{q}}_{\alpha})+(p_{(\alpha,\overline{\overline{g}}_{\alpha},\overline{\overline{q}}_{\alpha})}-p_{(\overline{\overline{g}},\overline{\overline{q}})}))_{Q}\geq
0,
\]
then, we deduce that
\begin{equation}
\|\overline{\overline{q}}-\overline{\overline{q}}_{\alpha}\|_{Q}\leq
\frac{\|\gamma_{0}\|}{M_{2}}\|p_{(\alpha,\overline{\overline{g}}_{\alpha},\overline{\overline{q}}_{\alpha})}-p_{(\overline{\overline{g}},\overline{\overline{q}})}\|_{V}.\label{3}
\end{equation}
Next, in similar way, taking $h=\overline{\overline{g}}$ and
$\eta=0$ in the optimality condition for the problem
(\ref{PControlalfa}) and $h=\overline{\overline{g}}_{\alpha}$ and
$\eta=0$ in the optimality condition for the problem
(\ref{PControl}), we prove that
\begin{equation}
\|\overline{\overline{g}}-\overline{\overline{g}}_{\alpha}\|_{H}\leq
\frac{1}{M_{2}}\|p_{(\alpha,\overline{\overline{g}}_{\alpha},\overline{\overline{q}}_{\alpha})}-p_{(\overline{\overline{g}},\overline{\overline{q}})}\|_{V}.\label{4}
\end{equation}
Now, from (\ref{udebil2}) and the following inequalities, for
$\alpha >1$,
\[
\lambda_{1}\|u_{(\alpha,\overline{\overline{g}}_{\alpha},\overline{\overline{q}}_{\alpha})}-u_{(\overline{\overline{g}},\overline{\overline{q}})}\|_{V}^{2}+(\alpha-1)
\|u_{(\alpha,\overline{\overline{g}}_{\alpha},\overline{\overline{q}}_{\alpha})}-u_{(\overline{\overline{g}},\overline{\overline{q}})}\|_{R}^{2}\leq
\]
\[
\leq
(g,u_{(\alpha,\overline{\overline{g}}_{\alpha},\overline{\overline{q}}_{\alpha})}-u_{(\overline{\overline{g}},\overline{\overline{q}})})_{H}-(q,u_{(\alpha,\overline{\overline{g}}_{\alpha},\overline{\overline{q}}_{\alpha})}-u_{(\overline{\overline{g}},\overline{\overline{q}})})_{Q}
\]
\[
-a(u_{(\overline{\overline{g}},\overline{\overline{q}})},u_{(\alpha,\overline{\overline{g}}_{\alpha},\overline{\overline{q}}_{\alpha})}-u_{(\overline{\overline{g}},\overline{\overline{q}})})
\]
the result (\ref{uconvfuerte}) holds. In similar way, from (\ref{pdebil2}) and the inequality
\[
\lambda_{1}\|p_{(\alpha,\overline{\overline{g}}_{\alpha},\overline{\overline{q}}_{\alpha})}-p_{(\overline{\overline{g}},\overline{\overline{q}})}\|_{V}^{2}\leq
\]
\[
\leq
(u_{(\alpha,\overline{\overline{g}}_{\alpha},\overline{\overline{q}}_{\alpha})}-z_{d},p_{(\alpha,\overline{\overline{g}}_{\alpha},\overline{\overline{q}}_{\alpha})}-p_{(\overline{\overline{g}},\overline{\overline{q}})})_{H}
-
\]
\[
-a(p_{(\overline{\overline{g}},\overline{\overline{q}})},p_{(\alpha,\overline{\overline{g}}_{\alpha},\overline{\overline{q}}_{\alpha})}-p_{(\overline{\overline{g}},\overline{\overline{q}})})
-\alpha
(p_{(\overline{\overline{g}},\overline{\overline{q}})},p_{(\alpha,\overline{\overline{g}}_{\alpha},\overline{\overline{q}}_{\alpha})}-p_{(\overline{\overline{g}},\overline{\overline{q}})})_{R}
\]
we obtain the limit (\ref{pconvfuerte}). Finally, from (\ref{uconvfuerte}), (\ref{pconvfuerte}) and the
estimations (\ref{3}) and (\ref{4}), the limit
(\ref{controlesconvfuerte}) holds.
\end{proof}

\begin{cor}
If $(\overline{\overline{g}},\overline{\overline{q}})$ and
$(\overline{\overline{g}}_{\alpha},\overline{\overline{q}}_{\alpha})$
are the unique solutions of the \emph{simultaneous distributed and
Neumann boundary} optimal control problems (\ref{PControl}) and
(\ref{PControlalfa}), respectively, then we have:
\[
\lim_{\alpha\rightarrow\infty}|J_{\alpha}(\overline{\overline{g}}_{\alpha},\overline{\overline{q}}_{\alpha})-
J(\overline{\overline{g}},\overline{\overline{q}})|=0.
\]

\begin{proof}
It follows from the definition of $J$, $J_{\alpha}$ and
the last theorem.
\end{proof}
\end{cor}

\section*{\large Acknowledgements}
This paper has been partially sponsored by the Project PIP No. 0534
from \linebreak CONICET-Univ. Austral (Rosario, Argentina) and
AFOSR-SOARD Grant FA 9550-14-1-0122. The authors would like to thank
two anonymous referees for their constructive comments which
improved the readability of the manuscript.



\begin{thebibliography}{1}

\bibitem{BEM} BEN BELGACEM, F., EL FEKIH, H. and METOUI, H. (2003) Singular
perturbation for the Dirichlet boundary control of elliptic
problems. ESAIM: M2AN \textbf{37}, 833-850.
\bibitem{BEN} BENSOUSSAN, A. (1974) Teor\'{i}a moderna de control \'{o}ptimo, Cuadern.
Inst. Mat. Beppo Levi \# \textbf{7}, Rosario.
\bibitem{BouTa} BOUKROUCHE, M. and TARZIA, D.A. (2007) On a convex
combination of solutions to elliptic variational inequalities.
Electronic Journal of Differential Equations \textbf{31}, 1-10.
\bibitem{CaJa} CARSLAW, H.S. and JAEGER, J.C. (1959) Conduction of heat in
solids. Clareudon Press, Oxford.
\bibitem{Ca} CASAS, E. (1986) Control of an elliptic problem with pointwise state
constraints. SIAM J. Control Optim. \textbf{24}, 1309-1318.
\bibitem{CaRay}  CASAS, E. and RAYMOND, J.P. (2006) Error estimates for the numerical
approximation of Dirichlet boundary control for semilinear elliptic
equations. SIAM J. Control Optim. \textbf{45} (5), 1586-1611.
\bibitem{GaTa1}  GARIBOLDI, C.M. and TARZIA, D. A. (2003) Convergence of distributed
optimal controls on the internal energy in mixed elliptic problems
when the heat transfer coefficient goes to infinity. Appl. Math.
Optim. \textbf{47}, 213-230.
\bibitem{GaTa2} GARIBOLDI, C.M. and TARZIA, D.A. (2008) Convergence of boundary optimal control problems
with restrictions in mixed elliptic Stefan-like problems. Adv. in
Diff. Eq. and Control Processes \textbf{1}(2), 113-132.
\bibitem{KiSt} KINDERLEHRER, D. and STAMPACCHIA, G. (1980) An introduction to
variational inequalities and their applications. Academic Press, New
York.
\bibitem{KMV} KIRCHNER, A., MEIDNER, D. and VEXLER, B. (2011) Trust region methods with hierarchical finite element models for PDE-constrained
optimization. Control and Cybernetics \textbf{40} (4), 213-230.
\bibitem{Li} LIONS, J.L. (1968) C\^{o}ntrole optimal des systemes gouvern\'{e}s
par des \'{e}quations aux d\'{e}riv\'{e}es partielles.
Dunod-Gauthier Villars, Paris.
\bibitem{MiPu} MIGNOT, F. and PUEL, J.P. (1984) Optimal control in some variational
inequalities. SIAM J. Control Optim. \textbf{22}, 466-476.
\bibitem{TaTa}  TABACMAN, E.D. and TARZIA, D.A. (1989) Sufficient and/or necessary
condition for the heat transfer coefficient on $\Gamma _{1}$ and the
heat flux on $\Gamma _{2}$ to obtain a steady-state two-phase Stefan
Problem. J. Diff. Eq. \textbf{77}, 16-37.
\bibitem{Ta1} TARZIA, D.A. (1979) Sur le probl\`{e}me de Stefan \`{a} deux
phases. C.R.Acad. Sc. Paris \textbf{288}A, 941-944.
\bibitem{Tr} TR\"{O}LTZSCH, F. (2010) Optimal Control of Partial Differential
Equations. Theory, Methods  and Applications. American Mathematical
Society, Providence, Rhode Island.

\end{thebibliography}
\end{document}